\documentclass[12pt, reqno]{amsart}
\usepackage{amsmath, amsthm, amscd, amsfonts, amssymb, graphicx, color}
\usepackage[bookmarksnumbered, colorlinks, plainpages]{hyperref}
\hypersetup{colorlinks=true,linkcolor=red, anchorcolor=green, citecolor=cyan, urlcolor=red, filecolor=magenta, pdftoolbar=true}

\newtheorem{theorem}{Theorem}[section]
\newtheorem{lemma}[theorem]{Lemma}

\newtheorem{corollary}[theorem]{Corollary}
\theoremstyle{definition}
\newtheorem{definition}[theorem]{Definition}
\newtheorem{example}[theorem]{Example}

\newtheorem{prop}[theorem]{Proposition}

\theoremstyle{remark}
\newtheorem{remark}[theorem]{Remark}
\numberwithin{equation}{section}
\begin{document}

\setcounter{page}{1}

\title[On some geometric properties of operator spaces]{On some geometric properties of operator spaces}

\author[Arpita Mal, Debmalya Sain \MakeLowercase{and} Kallol Paul]{Arpita Mal,$^1$ Debmalya Sain,$^2$ \MakeLowercase{and} Kallol Paul$^1$$^{*}$}

\address{$^{1}$Department of Mathematics, Jadavpur University, Kolkata 700032, West Bengal, India.}
\email{\textcolor[rgb]{0.00,0.00,0.84}{arpitamalju@gmail.com;
kalloldada@gmail.com}}

\address{$^{2}$Department of Mathematics, Indian Institute of Science, Bengaluru 560012, Karnataka, India.}
\email{\textcolor[rgb]{0.00,0.00,0.84}{saindebmalya@gmail.com}}


\let\thefootnote\relax\footnote{Research of the first author is supported by UGC, Govt. of India. The research of Dr. Debmalya Sain is sponsored by Dr. D. S. Kothari Postdoctoral Fellowship, under the mentorship of Professor Gadadhar Misra. Dr. Debmalya Sain feels elated to acknowledge the joyous presence of the lovable couple, Mr. Partha Chowdhury and Mrs. Trina Das, in his life! The research of the  third author is supported by project MATRICS (Ref. no. MTR/2017/000059), SERB, DST, Govt. of India. }



\date{Received: xxxxxx; Revised: yyyyyy; Accepted: zzzzzz.
\newline \indent $^{*}$Corresponding author}

\begin{abstract}
In this paper we study some geometric properties like  parallelism, orthogonality and semi-rotundity in the space of bounded linear operators. We completely characterize parallelism of two compact linear operators between normed linear spaces $\mathbb{X} $ and $\mathbb{Y}$, assuming $\mathbb{X}$ to be reflexive. We also characterize parallelism of two bounded linear operators between normed linear spaces $\mathbb{X} $ and $\mathbb{Y}.$ We  investigate  parallelism and approximate parallelism in the space of bounded linear operators defined on a Hilbert space. Using the characterization of operator parallelism, we study Birkhoff-James orthogonality in the space of compact linear operators as well as bounded linear operators. Finally, we introduce the concept of semi-rotund points (semi-rotund spaces)   which generalizes the notion of exposed points (strictly convex spaces). We further study  semi-rotund operators and prove that $\mathbb{B}(\mathbb{X},\mathbb{Y})$ is a semi-rotund space which is not strictly convex, if $\mathbb{X},\mathbb{Y}$ are finite-dimensional Banach spaces and $\mathbb{Y}$ is strictly convex.  
\end{abstract}

\maketitle
\section{Introduction.} 

In this paper, letters $ \mathbb{X}, \mathbb{Y} $ denote normed linear spaces and $\mathbb{H}$ denotes a Hilbert space over the field $K\in\{\mathbb{C},~\mathbb{R}\}$. Also $\mathbb{X}^*$ denotes the dual space of $\mathbb{X}$.  Let $B_{\mathbb{X}} = \{x \in \mathbb{X} \colon \|x\| \leq 1\}$ and
$S_{\mathbb{X}} = \{x \in \mathbb{X} \colon \|x\|=1\}$ be the unit ball
and the unit sphere of $\mathbb{X}$ respectively. Let $\mathbb{B}(\mathbb{X}, \mathbb{Y})~ (\mathbb{K}(\mathbb{X}, \mathbb{Y})) $ denote the space of all bounded (compact) linear operators from $ \mathbb{X} $ to $ \mathbb{Y}. $  We write $\mathbb{B}(\mathbb{X}, \mathbb{Y})=\mathbb{B}(\mathbb{X}) ~ (\mathbb{K}(\mathbb{X}, \mathbb{Y})=\mathbb{K}(\mathbb{X})),$ if $\mathbb{X}=\mathbb{Y}$. For $x,y \in \mathbb{X}$, $x$ is said to be orthogonal to $y$ in the sense of Birkhoff-James \cite{B}, written as $x \perp_B y$, if $\|x\| \leq \|x+\lambda y\|$ for all $\lambda \in K.$ Moreover, $x$ is said to be norm-parallel \cite{Se}  to $y$, written as $x\parallel y$,   if $ \| x+\lambda y \|=\|x\|+\|y\|$ for some $\lambda \in \mathbb{T} $, where $\mathbb{T}=\{\lambda \in K: |\lambda|=1 \}$. We would like to note that in the context of a Banach space, Birkhoff-James orthogonality is homogeneous but not symmetric \cite{GSP}, whereas norm-parallelism is both  symmetric and $\mathbb{R}-$homogeneous. The notion of Birkhoff-James orthogonality coincides with inner product orthogonality if the underlying space is a Hilbert space. Furthermore, two elements of a Hilbert space are norm-parallel if and only if they are linearly dependent. In case of normed linear spaces, two linearly dependent vectors are norm-parallel, but the converse is not true in general. In $\ell_{\infty}^2$, $(1,1)$ and $(1,0)$ are norm-parallel but linearly independent. In \cite{PSJ}, Paul et al. studied the notion of strong orthogonality. For $x,y \in \mathbb{X}$, $x$ is said to be
 strongly orthogonal to $y$ in the sense of Birkhoff-James, written as $x \perp_{SB} y$, if $\|x\| < \|x+\lambda y\|$ for all $\lambda \in K\setminus \{0\}$. \\
 Birkhoff-James orthogonality plays a very important role  in the study of geometry of Banach spaces. It has been explored in various settings by many other mathematicians \cite{AR,BFS,BS,LS,BG}. Furthermore,  various generalizations of it have also been considered. Dragomir \cite{D}  defined approximate Birkhoff-James orthogonality  as follows:\\
Let $\epsilon \in [0,1)$ and $x,y\in \mathbb{X}$. Then $x$ is said to be approximate Birkhoff-James orthogonal to $y$ if $ \|x+\lambda y\|\geq (1-\epsilon) \|x\|~\forall ~\lambda \in \mathbb{K}.$\\
Later on, Chmieli\'nski \cite{C} slightly modified the definition given by Dragomir and defined approximate Birkhoff-James orthogonality  as follows:
\[x\perp_D^{\epsilon}y \Longleftrightarrow \|x+\lambda y\|\geq \sqrt{1-\epsilon^2} \|x\|~\forall ~\lambda \in \mathbb{K}.\]
Motivated by the notion of approximate Birkhoff-James orthogonality, Moslehian and Zamani \cite{MZ} introduced the notion  of approximate parallelism($\epsilon$-parallelism) in the setting of normed linear space. For $x,y \in \mathbb{X}$ and $ \epsilon \in [0,1),$ $  x$ is said to be approximate parallel to $y$, written as, $x \parallel^{\epsilon} y$, if $\inf\{\|x+\mu y\|:\mu \in K \}\leq \epsilon \|x\|.$ We would like to remark here that in general in a normed linear space, $x \parallel^{\epsilon} y$ with $\epsilon=0$ implies that $x\parallel y$ but not the other way round. As for example, in $\ell_{\infty}^2,~(1,1)\parallel (1,0)$ but $(1,1)\not\parallel^{\epsilon}(1,0)$ with $\epsilon=0.$\\

We will see later that the norm attainment set of a bounded linear operator plays an important role in the study of norm-parallelism and Birkhoff-James orthogonality of bounded linear operators between Banach spaces.
For a bounded linear operator  $T \in \mathbb{B}(\mathbb{X}, \mathbb{Y})$, we define $M_T$ to be the set of all unit vectors in $S_\mathbb{X}$
at which $T$ attains norm, i.e.,
\[
M_T = \{ x \in S_\mathbb{X} \colon \|Tx\| = \|T\| \}.
\] 

The purpose of this paper is to study norm-parallelism and Birkhoff-James orthogonality in the space of bounded  linear operators, from the point of view of operator norm attainment. Bottazzi et al.\cite[Th. 4.24]{BCMWZ} proved that if $\mathbb{X}$ is a reflexive Banach space, $\mathbb{Y}$ is a smooth, strictly convex Banach space, $T,A \in \mathbb{K}(\mathbb{X},\mathbb{Y})$ and $M_{T}$ is either connected or $M_{T}=\{\pm u\}$ for some unit vector $u \in \mathbb{X},$ then $T \parallel A$ if and only if there exists a vector $x\in M_{T}\cap M_{A}$ such that $T x \parallel Ax$. In this paper, we substantially improve upon the result to show that there is no necessity to put additional restrictions on the norm attainment set $M_T$ and the codomain space $\mathbb{Y}.$ We next give a complete characterization of norm-parallelism of bounded linear operators defined between any two normed linear spaces.

We further study and completely characterize norm parallelism in the space of bounded linear operators on a Hilbert space $\mathbb{H}.$ We also obtain a necessary condition for approximate parallelism in $\mathbb{B}(\mathbb{H})$ and provide an example to illustrate the subtle difference between norm parallelism of operators ($T\parallel A$) and approximate operator parallelism ($T\parallel^{\epsilon}A,$ with $\epsilon=0$).
\\

In section 3, we explore Birkhoff-James orthogonality in the space of bounded linear operators by classifying them into two exclusive and exhaustive categories. We first study the case $``T\perp_{SB}A$'' and obtain a characterization of strong Birkhoff-James orthogonality in the space of bounded linear operators between finite-dimensional normed linear spaces. We next study the case ``$T\perp_B A$ but $T\not\perp_{SB} A$'' and obtain a necessary condition for the same.

Motivated by the operator theoretic results involving Birkhoff-James orthogonality and strong Birkhoff-James orthogonality, we introduce a new geometric notion, that of a semi-rotund point, defined in the following way:
\begin{definition}[semi-rotund point]
Let $\mathbb{X}$ be a normed linear space. An element  $ \theta \neq x\in \mathbb{X}$ is said to be a  semi-rotund point of $\mathbb{X}$ if there exists $y\in \mathbb{X}$ such that $x\perp_{SB}y$.
\end{definition}
\begin{definition}
A normed linear space $\mathbb{X}$ is said to be a semi-rotund space if for each non-zero $x\in \mathbb{X}$, $x$ is a semi-rotund point.
\end{definition} 
Clearly, every exposed point of the unit ball of a normed linear space is a semi-rotund point. However, it is interesting to observe that the converse is not true if the dimension of the space is greater than two. In $\ell_{\infty}^{3}$, the point $(1,1,0)$ is a semi-rotund point but not an exposed point of the unit ball. It is also immediate that a strictly convex space is a semi-rotund space but the converse is not necessarily true. For example, consider the space $\mathbb{X}$, where $S_{\mathbb{X}}=\{(x,y,z)\in \mathbb{R}^3:x^2+y^2=1,|z|\leq 1\} \cup \{(x,y,z)\in \mathbb{R}^3:x^2+y^2+(z-1)^2=1,~1\leq z\leq 2\} \cup \{(x,y,z)\in \mathbb{R}^3:x^2+y^2+(z+1)^2=1,~-2\leq z\leq -1\} $. In this case, it is easy to verify that every non-zero element in $\mathbb{X}$ is a semi-rotund point but not every point of $ S_{\mathbb{X}} $ is an exposed point of $ B_{\mathbb{X}} $. We note that the notions of exposed point (strictly convex space) and semi-rotund point (semi-rotund space) are equivalent if the space is two-dimensional. Continuing our study of semi-rotund points, we prove that every non-zero compact linear operator from a reflexive Banach space to a strictly convex Banach space is a semi-rotund point of the corresponding operator space. Finally, we show that $\mathbb{B}(\mathbb{X},\mathbb{Y})$ is a semi-rotund space which is not strictly convex, if $\mathbb{X}$ and $\mathbb{Y}$ are finite-dimensional Banach spaces and in addition, $\mathbb{Y}$ is strictly convex. In particular, this illustrates that the concept of semi-rotundity is a proper generalization of the concept of strict convexity. It is well-known that several convexity conditions such as uniform convexity and local uniform rotundity, strictly stronger than that of strict convexity, are of great importance in the study of the geometry of normed linear spaces. We would like to end this section with the remark that to the best of our knowledge, this is the first instance of a convexity property in normed linear spaces which is strictly weaker than strict convexity.    

\section{Norm-parallelism of bounded linear operators}
We begin this section with an easy proposition on approximate parallelism.   
\begin{prop}
	Let $\mathbb{X},~\mathbb{Y}$ be two normed linear spaces. Let $T \in\mathbb{B}({\mathbb{X},\mathbb{Y}})$ and $x \in M_T$. Then for any $\epsilon \in [0,1)$ and any $y \in \mathbb{X}$, $x \parallel^{\epsilon} y \Longrightarrow Tx \parallel^{\epsilon} Ty.$ 
\end{prop}
\begin{proof}
	Let $x \parallel^{\epsilon}y$. Then $\inf\{\|x+\lambda y\|~:~\lambda \in K\} \leq \epsilon \|x\|$. Thus, $\inf\{\|Tx+\lambda Ty\|~:~\lambda \in K\} \leq \|T\| \inf\{\|x+\lambda y\|~:~\lambda \in K\} \leq \epsilon \|T\|\|x\|=\epsilon \|Tx\|$. Therefore, $Tx \parallel^{\epsilon}Ty$. 
\end{proof}

In \cite[Th. 2.1]{J}, James characterized Birkhoff-James orthogonality in terms of linear functionals. We next state a lemma that characterizes approximate parallelism in terms of linear functionals. We would like to remark that the lemma follows from a slight variation of  \cite[Cor. 6.8]{Con}, and therefore its proof is omitted.
\begin{lemma}\label{theorem:epsilonparallel}
	Let $\mathbb{X}$ be a normed linear space. Let $x,~y \in \mathbb{X}$ and $d=\inf\{\|x+\lambda y\|:~\lambda \in K\}$. Then for any $\epsilon \in [0,1)$, $x \parallel^{\epsilon}y$ if and only if there exists a linear functional $f \in S_{\mathbb{X}^*}$ such that $f(x)=d\leq \epsilon \|x\|$ and $f(y)=0$.
\end{lemma}

 Now, we obtain a  complete characterization of norm-parallelism of compact linear operators defined on a reflexive Banach space, which substantially improves on \cite[Th. 4.24]{BCMWZ}.
\begin{theorem}\label{theorem:compact}
 Let $\mathbb{X}$ be a reflexive Banach space and $\mathbb{Y}$ be any normed linear space. Let $T,~A\in \mathbb{K}(\mathbb{X},\mathbb{Y})$.
Then $T\parallel A$ if and only if there exists $x \in M_{T} \cap M_{A}$ such that $Tx \parallel  Ax$.
\end{theorem}
\begin{proof}
	First we prove the necessary part of the theorem. Let $T \parallel A$. Then there exists $\lambda \in \mathbb{T}$ such that $\|T+ \lambda A\|= \|T\|+\|A\|$. The operator $T + \lambda A$, being a compact operator on a reflexive Banach space, attains its norm. Therefore, there exists $x \in S_{\mathbb{X}}$ such that $\|T+\lambda A\|=\|(T+\lambda A)x\|$. Thus,
	\begin{eqnarray*}
	      \|T\|+\|A\|&=&\|T+\lambda A\|\\
	      &=&\|(T+\lambda A)x\|\\
	      &\leq& \|Tx\|+\|Ax\|\\
	      &\leq& \|Tx\|+\|A\|\\
	      &\leq& \|T\|+\|A\|.
	\end{eqnarray*}
	 This implies that $\|Tx+\lambda A x\|=\|Tx\|+\|Ax\|$ and $\|T x\|=\|T\|$. Similarly, $\|Ax\|=\|A\|$. Therefore, $x\in M_{T}\cap M_{A}$ and $Tx \parallel Ax$. This completes the proof of the necessary part of the theorem.\\
	For the sufficient part, suppose that  there exists $x \in M_{T} \cap M_{A}$ such that $Tx \parallel Ax$. Then there exists $\lambda \in \mathbb{T}$ such that $\|Tx+\lambda Ax\|=\|Tx\|+\|Ax\|$. Therefore, 
	\[\|T\|+\|A\|\geq \|T+\lambda A\|\geq \|(T+\lambda A)x\|=\|Tx\|+\|Ax\|=\|T\|+\|A\|.\]
	 Thus, $\|T+\lambda A\|=\|T\|+\|A\|$, i.e., $T\parallel A$. This completes the proof of the theorem.
\end{proof}

We make note of the following remark that will be needed later in Theorem \ref{theorem:compactortho}.

\begin{remark}\label{remark:compact}
From the proof of Theorem \ref{theorem:compact}, it is clear that $\|T+\lambda A\|=\|T\|+\|A\|$ for $\lambda \in \mathbb{T}$ if and only if there exists $x\in M_{T}\cap M_{A}$ such that $\|Tx+\lambda Ax\|=\|Tx\|+\|Ax\|$. 
\end{remark}

We next give an example to show that the compactness of $T,~A$ in Theorem \ref{theorem:compact} is essential.

\begin{example}
	 Consider the right shift operator $T:\ell_2\longrightarrow \ell_2$ defined by $T(x_1,x_2,$\\ $x_3,\ldots)=(0,x_1,x_2,x_3,\ldots)$. Let $ I $ be the identity operator on $ \ell_2. $ Then $\|T\|=\|I\|=1.$ Consider $y_n=\frac{1}{\sqrt{n}}(\underbrace{1,1,\ldots,1}_{n},0,0,\ldots)\in S_{\ell_2}$ for each $n \in \mathbb{N}$. Then $\|(T+I)y_n\|^2=\frac{1}{n}\|(1,2,2,\ldots,2,1,0,0,\ldots)\|^2=\frac{1}{n}[2+4(n-1)]\longrightarrow 4$. Thus $2\leq \|T + I\| \leq \|T\|+\|I\| =  2 \Rightarrow \|T+I\|=\|T\|+\|I\|$, i.e., $T\parallel I$. We claim that there does not exist any $x\in M_T\cap M_I$ such that $Tx\parallel Ix$. For, if $Tx\parallel Ix$ for some $x=(x_1,x_2,x_3,\ldots)\in M_T \cap M_I$ then there exists $\lambda \in \mathbb{T}$ such that $\|Tx+\lambda Ix\|=\|Tx\|+\|Ix\|$ and so  $ \|Tx\|\|Ix\| = Re\{\overline{\lambda}\langle Tx,Ix\rangle\}\leq |\overline{\lambda}\langle Tx,Ix\rangle|=| \langle Tx,Ix\rangle| \leq \|Tx\|\|Ix\|.$ 	Then   $ | \langle Tx,Ix\rangle| = \|Tx\|\|Ix\|=\|T\|\|I\|,$ since $x \in M_T.$
	Then $  |0\overline x_1+x_1 \overline x_2 +x_2 \overline x_3 + \ldots |=1=(0^2+|x_1|^2+|x_2|^2+\ldots)^{\frac{1}{2}}(|x_1|^2+|x_2|^2+|x_3|^2+\ldots)^{\frac{1}{2}}$. Thus, by the equality condition of Cauchy-Schwarz inequality, we have, $(0,x_1,x_2,x_3,\ldots)=\lambda(x_1,x_2,x_3,\ldots)$ for some $\lambda \in K$, which implies that $x=0$, a contradiction. 
\end{example}

 In the next theorem, we characterize norm-parallelism of bounded linear operators between any two normed linear spaces. 
\begin{theorem}\label{theorem:bounded}
	Let $\mathbb{X},~\mathbb{Y}$ be two normed linear spaces and $T,~A \in \mathbb{B}(\mathbb{X},\mathbb{Y})$. Then $T\parallel A$ if and only if there exists a sequence $\{x_n\}$ in $S_{\mathbb{X}}$ such that $$\lim_{n \to \infty}\|Tx_n\|=\|T\|,~\lim_{n \to \infty}\|Ax_n\|=\|A\|$$ and $$\lim_{n\to\infty}\|Tx_n+\lambda Ax_n\|=\|T\|+\|A\|$$ for some $\lambda \in \mathbb{T}$.
\end{theorem}
\begin{proof}
	First we prove the necessary part of the theorem. Let $T\parallel A$. Then there exists $\lambda \in \mathbb{T}$ such that $\|T+ \lambda A\|= \|T\|+\|A\|$. Now, there exists a sequence $\{x_n\}$ in $S_{\mathbb{X}}$ such that $\lim_{n\to\infty}\|(T+\lambda A)x_n\|=\|T+\lambda A\|$. Since $\{\|Tx_n\|\}$ and $\{\|Ax_n\|\}$ are bounded sequences of real numbers, without loss of generality (if necessary,  passing onto a subsequence) we can assume that $\lim_{n\to\infty}\|Tx_n\|$ and $\lim_{n\to\infty}\|Ax_n\|$ exists. Therefore, 
	\begin{eqnarray*}
         \|T\|+\|A\|&=&\|T+\lambda A\|\\
               &=&\lim_{n\to\infty}\|(T+\lambda A)x_n\|\\
               &\leq &\lim_{n\to\infty}\|Tx_n\|+\lim_{n\to\infty}\|Ax_n\|\\
               &\leq &\lim_{n\to\infty}\|Tx_n\|+\|A\|\\
               &\leq &\|T\|+\|A\|.
	\end{eqnarray*}
	 This implies that $\lim_{n\to\infty}\|Tx_n+\lambda A x_n\|=\lim_{n\to\infty}\|Tx_n\|+\lim_{n\to\infty}\|Ax_n\|$ and $\lim_{n\to\infty}\|T x_n\|=\|T\|$. Similarly, $\lim_{n\to\infty}\|A x_n\|=\|A\|$. This completes the proof of the necessary part of the theorem.\\
	For the sufficient part of the theorem, assume that  there exists a sequence $\{x_n\}$ in $S_{\mathbb{X}}$ such that $\lim_{n \to \infty}\|Tx_n\|=\|T\|,~\lim_{n \to \infty}\|Ax_n\|=\|A\|$ and $\lim_{n\to\infty}\|Tx_n+\lambda A x_n\|=\|T\|+\|A\|$ for some $\lambda \in \mathbb{T}$. Therefore, 
	\begin{eqnarray*}
	   \|T\|+\|A\|&\geq &\|T+\lambda A\|\\
	   &\geq &\lim_{n\to\infty}\|(T+\lambda A)x_n\|\\
	   &=&\|T\|+\|A\|.
	\end{eqnarray*}
   Thus, $\|T+\lambda A\|=\|T\|+\|A\|$, i.e., $T\parallel A$. This completes the proof of the theorem.
\end{proof}
\begin{remark}
We note that norm parallelism in operator space can also be characterized by \cite[Th. 2.4]{ZM} using the notion of Birkhoff-James orthogonality. 
\end{remark}

We next give an easy characterization of strictly convex spaces in terms of norm parallelism.  

\begin{theorem}\label{theorem:strictlyconvex}
A normed linear space $\mathbb{X}$ is strictly convex if and only if for any $x,y\in \mathbb{X}$, $x\parallel y \Leftrightarrow \{x,y\}$ is linearly dependent. 
\end{theorem}
\begin{proof}
Let $\mathbb{X}$ be strictly convex. Let $\{x,y\}$ be linearly dependent. Let $y=\alpha x.$ If $\alpha =0,$ then clearly,$\|x+y\|=\|x\|+\|y\|.$ Let $\alpha\neq 0.$ Let $\lambda=\frac{\overline{\alpha}}{\mid \alpha \mid}.$ Then $\lambda \in \mathbb{T}$ and $\|x+\lambda y\|=\|x+\frac{\overline{\alpha}}{\mid \alpha \mid}~\alpha x\|=\|x+ | \alpha | x\|=\|x\|+\|y\|.$ Thus, $x\parallel y$. On the other hand, let $x\parallel y$. Then there exists $\lambda \in \mathbb{T}$ such that $\|x+\lambda y\|=\|x\|+\|y\|=\|x\|+\|\lambda y\|$. Since $\mathbb{X}$ is strictly convex, $\{x,y\}$ is linearly dependent. \\
Conversely, suppose that $\mathbb{X}$ is not strictly convex. Then there exists two linearly independent vectors $x,y\in S_{\mathbb{X}}$ and $t\in (0,1)$ such that $\|(1-t)x+ty\|=1=\|(1-t)x\|+\|ty\|.$ Therefore, it follows from the homogeneity property of parallelism that  $x\parallel y$. This completes the proof of the theorem.
\end{proof}

In \cite[Th. 4.14]{BCMWZ}, Bottazzi et al. proved that if $\mathbb{X}$ is a locally uniformly convex Banach space and $A\in \mathbb{K}(\mathbb{X})$ is such that $A^{m-1}\neq 0$ and $A^m=0$ for some $m\in \mathbb{N}$ then $A^k \not\parallel A^j$ for every $1\leq k <j <m$. It turns out that the condition ``$\mathbb{X}$ is a locally uniformly convex Banach space'' is redundant and as a matter of fact, we prove the theorem under the weaker assumption of strict convexity. We further show that strict convexity is essential for the result to hold true.

\begin{theorem}\label{theorem:nilpotent}
	Let $\mathbb{X}$ be a strictly convex Banach space. Let $A \in \mathbb{K}(\mathbb{X})$ be such that $A^m=0$ and $A^j\neq 0$ for all $1\leq j <m$. Then $A^k  \nparallel  A^j$ for every $1\leq k<j<m$.
\end{theorem}
\begin{proof}
	Let $1\leq k<j<m$. If possible, suppose that $A^k \parallel A^j$. Then from Theorem \ref{theorem:bounded}, we have, there exists $\{x_n\}$ in $S_{\mathbb{X}}$ such that $$\lim_{n\to\infty}\|A^k x_n\|=\|A^k\|,~\lim_{n\to\infty}\|A^j x_n\|=\|A^j\|$$ and $$\lim_{n\to\infty}\|A^k x_n+ \lambda A^j x_n\|=\|A^k\|+\|A^j\|,$$ for some $\lambda \in \mathbb{T}$. Since $A\in \mathbb{K}(\mathbb{X})$, it follows that $A^k,~A^j\in \mathbb{K}(\mathbb{X})$. Therefore, $\{A^k x_n\}$ has a convergent subsequence say, $\{A^k x_{n_i}\}$ converging to $y \in \mathbb{X}$. Hence $\|A^k x_{n_i}\|$ converges to $\|y\|$. Thus, $\|y\|=\|A^k\|$. Again, $\lim_{n\to\infty}(A^k x_{n_i})= y$ implies that $A^j x_{n_i}=A^{j-k}(A^k x_{n_i})$ converges to $A^{j-k}y$. Therefore, $\|A^j x_{n_i}\|$ converges to $\|A^{j-k}y\|$ and hence $\|A^{j-k}y\|=\|A^j\|$. Now, $A^k x_{n_i}+\lambda A^j x_{n_i}$ converges to $y +\lambda A^{j-k}y$ implies that  $\|A^k x_{n_i}+\lambda A^j x_{n_i}\|$ converges to $\|y +\lambda A^{j-k}y\|$. Therefore, $\|y+\lambda A^{j-k}y\|=\|A^k\|+\|A^j\|=\|y\|+\|A^{j-k}y\|$. Hence $y \parallel A^{j-k}y$. Since $\mathbb{X}$ is strictly convex, $A^{j-k}y=\alpha y$ for some $\alpha \in K$. This implies that $A^{m(j-k)}y=\alpha^m y$. Thus, $\alpha^m y=0$, since $A^m=0$. Therefore, either $y=0$ or $\alpha =0$. Now, $y=0$ gives that $A^k=0$ and $\alpha =0$ gives that $A^{j-k}y=0$, i.e., $A^j=0.$ Thus, in any case, we reach a  contradiction to the hypothesis. Therefore, $A^k \nparallel A^j$. This completes the proof of the theorem.    
\end{proof}

We now give an example to show that in Theorem \ref{theorem:nilpotent}, strict convexity of $\mathbb{X}$ is essential.
\begin{example}
 Let $\mathbb{X}=\ell_1^3$. Define $A\in \mathbb{B}(\mathbb{X})$ by 
	\[A(1,0,0)=-(0,1,0)\]
	\[A(0,1,0)=-(0,0,1)\]
	\[A(0,0,1)=(0,0,0).\]
	Clearly, $A\neq 0,~A^2 \neq 0$ and $A^3=0$. It is easy to observe that $(1,0,0)\in M_A \cap M_{A^2}$ and $A(1,0,0) \parallel A^2(1,0,0)$. Therefore, by Theorem \ref{theorem:compact}, we have $A \parallel A^2$. 
\end{example}

In  \cite[Th. 4.15]{BCMWZ}, Bottazzi et al. investigated norm-parallelism of idempotent operators defined on a locally uniformly convex Banach space. In the next theorem, we study the problem when the underlying space is strictly convex.
\begin{theorem}\label{theorem:idempotent}
	Let $\mathbb{X}$ be a strictly convex normed linear space. Let $A,~B \in \mathbb{K}(\mathbb{X})$ be such that $A \neq 0,~B \neq 0,~A^2=A$ and $B^2=B$. If $A \parallel B$ then $A(\mathbb{X})\cap B(\mathbb{X})\neq \{0\}$.
\end{theorem}
\begin{proof}
	Suppose $A \parallel B$. Then from Theorem \ref{theorem:bounded}, we have, there exists a sequence $\{x_n\}$ in $S_{\mathbb{X}}$ such that $\lim_{n\to\infty}\|Ax_n\|=\|A\|,~\lim_{n\to\infty}\|Bx_n\|=\|B\|$ and $\lim_{n\to\infty}\|Ax_n+\lambda Bx_n\|=\|A\|+\|B\|$ for some $\lambda \in \mathbb{T}$. Since $A$ and $B$ are compact operators, $\{Ax_n\}$ and $\{Bx_n\}$ have convergent subsequences. Without loss of generality we assume that $Ax_n \longrightarrow y$ and $Bx_n \longrightarrow z$. Therefore, $A^2x_n \longrightarrow Ay$. Since $A^2=A$, we have, $Ay=y$. Similarly $Bz=z$. Again, $Ax_n \longrightarrow y$ and $\|Ax_n\| \longrightarrow \|A\|$ implies that $\|A\|=\|y\|$. Similarly, $\|B\|=\|z\|$. Clearly, $Ax_n+\lambda Bx_n \longrightarrow y+\lambda z$. Therefore, $\|y+\lambda z\|=\|A\|+\|B\|=\|y\|+\|z\|$. Thus, $y \parallel z$. Since $\mathbb{X}$ is strictly convex, $z=\alpha y$ for some $\alpha \in K$. Now, $Ay=y,~Bz=z$ and $z=\alpha y$ gives that $z\in A(\mathbb{X})\cap B(\mathbb{X})$. Clearly, $z\neq 0$. For, otherwise $\|z\|=\|B\|=0$ implies that $B=0$, a contradiction to the hypothesis. Therefore, $A(\mathbb{X})\cap B(\mathbb{X})\neq \{0\}$. This completes the proof of the theorem.     
\end{proof}

The following example shows that in Theorem \ref{theorem:idempotent}, strict convexity of $\mathbb{X}$ is essential.
\begin{example}
	 Let $\mathbb{X}=\ell_1^3$ and $A,~B \in \mathbb{B}(\ell_1^3)$ be given by the following matrices (with respect to the standard ordered basis of $ \mathbb{R}^{3}) $ $$
	\begin{bmatrix}
	1&0&0\\
	0&0&0\\
	0&0&0
	\end{bmatrix}
	,
	\quad
	\begin{bmatrix}
	0&0&0\\
	0&1&0\\
	1&0&1
	\end{bmatrix}
	$$ respectively. Clearly, $A^2=A$ and $B^2=B$. It is easy to verify that $(1,0,0)\in M_A\cap M_B$ and $A(1,0,0)\parallel B(1,0,0)$. Therefore, by Theorem \ref{theorem:compact}, we have, $A \parallel B$. Clearly, in this case, $A(\mathbb{X})\cap B(\mathbb{X})=\{(0,0,0)\}$.
\end{example}

Next, we study approximate parallelism in the space of bounded linear operators on an infinite-dimensional Hilbert space.  
\begin{theorem}\label{theorem:epsilon}
	Let $\mathbb{H}$ be an infinite-dimensional Hilbert space. Let $\epsilon \in [0,1)$ and $T\in \mathbb{B}(\mathbb{H}).$ Then $(i)\Rightarrow (ii),$ where,\\
	(i) For any $A\in\mathbb{B}(\mathbb{H})$, $T\parallel^{\epsilon} A \Leftrightarrow$ there exists $ x \in M_T\cap M_A$ such that $Tx \parallel^{\epsilon} Ax$.\\
	(ii) There exists a finite-dimensional subspace $H_0$ of $\mathbb{H}$ such that $M_T=S_{H_0}$ and $\|T\|_{H_0^{\perp}}<\|T\|$. 
\end{theorem}
\begin{proof}
	Without loss of generality assume that $\|T\|=1$. From \cite[Th. 2.2]{SP}, we have, $M_T=S_{H_0}$, where $H_0$ is a subspace of $\mathbb{H}$. We first show that $H_0$ is finite-dimensional. If possible, suppose that $H_0$ is infinite-dimensional. Then there exists a sequence $\{e_n:n\in \mathbb{N}\}$ of orthonormal vectors in $H_0$. Extend this sequence to a complete orthonormal basis $\mathcal{B}=\{e_\alpha:\alpha \in \Lambda\supseteq \mathbb{N}\}$ of $\mathbb{H}.$  For each $e_\alpha \in H_0 \cap \mathcal{B} $ we have 
$ \| T^{*}T \| = \|T\|^2 = \|Te_\alpha \|^2 = \langle T^{*}T e_\alpha , e_\alpha \rangle \leq \|T^{*}Te_\alpha \| \|e_\alpha \| \leq \|T^{*}T\|, $
so that by the equality condition of Schwarz's inequality, we get $ T^{*}Te_\alpha = \lambda_\alpha e_\alpha $ for some scalar $\lambda_\alpha.$ 
Thus, $\{ Te_\alpha : e_\alpha \in H_0 \cap \mathcal{B}\}$ is a set of orthonormal vectors in $\mathbb{H}.$
 Define $A:\mathcal{B}\longrightarrow \mathbb{H}$ as follows:
	\begin{eqnarray*}
		Ae_n&=&\frac{1}{n^2}Te_n,~\forall n \in \mathbb{N}\\
		Ae_\alpha&=&0,~\forall~ e_\alpha \in  \mathcal{B} \setminus \{e_n:n\in \mathbb{N}.\}
			\end{eqnarray*}
	Since $\{Te_\alpha:\alpha \in H_0\cap \mathcal{B}\}$ is orthonormal, $A$ can be extended to a bounded linear operator on $\mathbb{H}$. Now, for any $\lambda \in K,~\|T+\lambda A\|\geq \|(T+\lambda A)e_n\|=|1+\frac{\lambda}{n^2}|\|Te_n\|=|1+\frac{\lambda}{n^2}|\longrightarrow 1$. Therefore, $\inf_{\lambda \in K}\|T+\lambda A\|\geq 1>\epsilon $. This implies that $T \not \parallel^\epsilon A$. But $e_1 \in M_T\cap M_A$ and clearly $Te_1 \parallel^{\epsilon}Ae_1$. However, this clearly contradicts the hypothesis. Therefore, $H_0$ must be finite-dimensional. \\
	Next, we show that $\|T\|_{H_0^{\perp}}< 1$. If possible, suppose that $\|T\|_{H_0^{\perp}}= 1$. Then there exists a sequence $\{x_n\}$ in $S_{H_0^{\perp}}$ such that $\|Tx_n\|\longrightarrow 1$. Define $A:\mathbb{H}\longrightarrow \mathbb{H}$ by $Az=Tx$, where $z=x+y,~x\in H_0,~y\in H_0^{\perp}.$
	It is easy to verify that $A\in \mathbb{B}(\mathbb{H})$. Now, for any $\lambda \in K$, $\|T+\lambda A\|\geq \|(T+\lambda A)x_n\|=\|Tx_n\|\longrightarrow 1$. Therefore, $\inf_{\lambda\in K}\|T+\lambda A\|\geq 1>\epsilon $. This shows that $T\not\parallel^\epsilon A$ but clearly, for any $x\in M_T\cap M_A$, $Tx\parallel^\epsilon Ax$. Once again, this contradicts the hypothesis. Therefore, we must have, $\|T\|_{H_0^{\perp}}<1$. This establishes the theorem.  
\end{proof}

The following example shows that the condition $(ii)$ of Theorem \ref{theorem:epsilon} does not imply condition $(i)$ of Theorem \ref{theorem:epsilon}, i.e., there exists $T,A  \in \mathbb{B}(\mathbb{H})$ and a finite-dimensional subspace $H_0$ of $\mathbb{H}$ such that $M_T=S_{H_0}$ and $\|T\|_{H_0^{\perp}}<\|T\|$ but $T\parallel^{\epsilon} A \Leftrightarrow Tx \parallel^{\epsilon} Ax$ for some $x\in M_T\cap M_A$ does not hold.
\begin{example}
  Consider $T,A\in \mathbb{B}(\ell_2)$ defined by 
	\begin{eqnarray*}
	T(x_1,x_2,x_3,\ldots,x_n,\ldots)&=&(x_1,\frac{x_2}{2},\frac{x_3}{2},\ldots,\frac{x_n}{2},\ldots),\\
	A(x_1,x_2,x_3,\ldots,x_n,\ldots)&=&(x_1,0,0,\ldots,0,\ldots),
	\end{eqnarray*}
	respectively, where $(x_1,x_2,x_3,\ldots,x_n,\ldots)\in \ell_2$. It is easy to verify that $M_T=S_{H_0}$, where $H_0=span\{(1,0,0,\ldots)\}$  and $\|T\|_{H_0^{\perp}}=\frac{1}{2}<\|T\|$. Let $\epsilon \in [0,\frac{1}{2})$. Clearly, $(1,0,0,\ldots)\in M_T \cap M_A$ and  $T(1,0,0,\ldots)\parallel^{\epsilon}A(1,0,0,\ldots)$. Now, since $$\langle A(0,1,0,0,\dots,0,\ldots),(0,1,0,0,\ldots,0,\ldots)\rangle=0$$ and $$\langle T(0,1,0,0,\dots,0,\ldots),(0,1,0,0,\ldots,0,\ldots)\rangle=\frac{1}{2},$$ $\sup\{|\langle T\xi,\eta \rangle|:\|\xi\|=\|\eta\|=1,~\langle A\xi, \eta \rangle =0\}\geq \frac{1}{2}>\epsilon$. Therefore, by \cite[Th. 3.7]{MZ}, we have, $T\not\parallel^\epsilon A$. 
\end{example}

However, if we consider $T\parallel A$ instead of $T\parallel^{\epsilon}A~(\epsilon\in [0,1))$ then we have the following characterization. 

\begin{theorem}\label{theorem:hil}
	Let $\mathbb{H}$ be an infinite-dimensional Hilbert space. Let $T\in \mathbb{B}(\mathbb{H})$. Then the following two conditions are equivalent.\\
	(i) For any $A\in\mathbb{B}(\mathbb{H})$, $T\parallel A \Leftrightarrow$ there exists $x \in M_T \cap M_A$ such that $Tx \parallel Ax$. \\
	(ii) There exists a finite-dimensional subspace $H_0$ of $\mathbb{H}$ such that $M_T=S_{H_0}$ and $\|T\|_{H_0^{\perp}}<\|T\|$. 
\end{theorem}
\begin{proof} 
  $(ii)\Rightarrow (i)$ follows from  \cite[Th. 2.18]{ZM}.\\
	We only prove $(i)\Rightarrow (ii).$ Without loss of generality assume that $\|T\|=1$. From \cite[Th. 2.2]{SP}, we have, $M_T=S_{H_0}$, where $H_0$ is a subspace of $\mathbb{H}$. We first show that $H_0$ is finite-dimensional. If possible, suppose that $H_0$ is infinite-dimensional. Then there exists a sequence $\{e_n:n\in \mathbb{N}\}$ of orthonormal vectors in $H_0$. Extend this sequence to a complete orthonormal basis $\mathcal{B}=\{e_\alpha:\alpha \in \Lambda\supseteq \mathbb{N}\}$ of $\mathbb{H}.$  For each $e_\alpha \in H_0 \cap \mathcal{B} $ we have 
$ \| T^{*}T \| = \|T\|^2 = \|Te_\alpha \|^2 = \langle T^{*}T e_\alpha , e_\alpha \rangle \leq \|T^{*}Te_\alpha \| \|e_\alpha \| \leq \|T^{*}T\|, $
so that by the equality condition of Schwarz's inequality we get $ T^{*}Te_\alpha = \lambda_\alpha e_\alpha $ for some scalar $\lambda_\alpha.$ 
Thus, $\{ Te_\alpha : e_\alpha \in H_0 \cap \mathcal{B}\}$ is a set of orthonormal vectors in $\mathbb{H}.$ Let $$\Lambda_1=\{\alpha \in \Lambda:e_\alpha \in (H_0\cap\mathcal{B})\setminus \{e_n:n\in \mathbb{N}\}\},~~\Lambda_2=\{\alpha \in \Lambda:e_\alpha \in \mathcal{B}\setminus H_0\}.$$ 
	If $ \Lambda_1 \neq \emptyset $, then $\Lambda_1$ can be well-ordered. Let $\alpha_0$ be the least element of $\Lambda_1$ and for any $\alpha \in \Lambda_1,~s(\alpha)$ be the successor of $\alpha$. If $\Lambda_1$ has a greatest element say, $\beta$, then define $s(\beta)=\alpha_0$. Now, define $A:\mathcal{B}\longrightarrow \mathbb{H}$ as follows:
	\begin{eqnarray*}
		Ae_n&=&Te_{n+1},~\forall ~ n\in \mathbb{N}\\
		Ae_\alpha&=& Te_{s(\alpha)},~ \forall ~ \alpha \in \Lambda_1\\
		Ae_\alpha&=&0,~\forall ~\alpha \in \Lambda_2.
	\end{eqnarray*}  
	Since  $\{Te_\alpha:e_\alpha \in H_0 \cap \mathcal{B}\}$ is a set of orthonormal vectors, $A$ can be extended to a bounded linear operator on $\mathbb{H}$.\\
	Next, we show that $T\parallel A$ but there does not exist any  $x\in M_T \cap M_A$ such that $Tx \parallel Ax$. Let $x=\Sigma_{n \in \mathbb{N}}\langle x,e_n \rangle e_n+\Sigma_{\alpha \in \Lambda_1}\langle x,e_\alpha \rangle e_\alpha+\Sigma_{\alpha \in \Lambda_2}\langle x,e_\alpha \rangle e_\alpha$. Then $Ax=\Sigma_{n \in \mathbb{N}}\langle x,e_n \rangle Te_{n+1}+\Sigma_{\alpha \in \Lambda_1}\langle x,e_\alpha \rangle Te_{s(\alpha)}$. Therefore, $\|Ax\|^2=\Sigma_{n \in \mathbb{N}}|\langle x,e_n \rangle|^2+\Sigma_{\alpha \in \Lambda_1}|\langle x,e_\alpha \rangle|^2\leq \|x\|^2$. Thus, $\|A\|\leq 1$. Again, $\|Ae_n\|=\|Te_{n+1}\|=1$ gives that $\|A\|=1$. Now, let $x_n=\frac{1}{\sqrt{n}}(e_1+e_2+\ldots +e_n)$. Then $x_n \in S_{\mathbb{H}}$ for all $n\in \mathbb{N}$. Now, $(T+A)x_n=\frac{1}{\sqrt{n}}(Te_1+2Te_2+2Te_3+\ldots +2Te_n+Te_{n+1})$. Therefore, $\|(T+A)x_n\|^2=\frac{1}{n}[2+4(n-1)]\Longrightarrow \|(T+A)x_n\|\longrightarrow 2$. Thus, we have, $2\leq \|T+A\|\leq \|T\|+\|A\|=2 \Longrightarrow \|T+A\|=\|T\|+\|A\|$. This gives that $T \parallel A$.\\
	Next, let $x\in M_T \cap M_A$. Let $x=\Sigma_{n \in \mathbb{N}}\langle x,e_n \rangle e_n+\Sigma_{\alpha \in \Lambda_1}\langle x,e_\alpha \rangle e_\alpha$. Let $n_0$ be the least positive integer such that $\langle x,e_{n_0} \rangle \neq 0$. Then $\|x\|=1\Longrightarrow \Sigma_{n=n_0}^{\infty}|\langle x,e_n \rangle|^2 +\Sigma_{\alpha \in \Lambda_1}|\langle x,e_\alpha \rangle|^2 =1 $. Now,
	\begin{eqnarray*}
		&&|\langle Tx, Ax\rangle|\\
		&=&|\langle \Sigma_{n=n_0}^{\infty}\langle x,e_n \rangle Te_n+\Sigma_{\alpha \in \Lambda_1}\langle x,e_\alpha \rangle Te_\alpha,\Sigma_{n =n_0}^{\infty}\langle x,e_n \rangle Te_{n+1}+\Sigma_{\alpha \in \Lambda_1}\langle x,e_\alpha \rangle Te_{s(\alpha)}\rangle|\\
		&=&|\Sigma_{n=n_0}^{\infty}\overline{\langle x,e_n\rangle}\langle x, e_{n+1}\rangle+ \Sigma_{\alpha \in \Lambda_1}\overline{\langle x,e_\alpha\rangle}\langle x, e_{s(\alpha)}\rangle|\\
		&\leq & \{\Sigma_{n=n_0}^{\infty}|\langle x,e_n\rangle|^2+\Sigma_{\alpha \in \Lambda_1}|\langle x,e_\alpha\rangle|^2\}^{\frac{1}{2}}\{\Sigma_{n=n_0+1}^{\infty}|\langle x,e_n\rangle|^2+\Sigma_{\alpha \in \Lambda_1}|\langle x,e_{s(\alpha)}\rangle|^2\}^{\frac{1}{2}}\\
		&<&1=\|Tx\|\|Ax\|.
	\end{eqnarray*}
	Thus, $ T \parallel A$ but there exists no $x \in M_T \cap M_A$ such that $ Tx \parallel Ax.$ \\	
	If $ \Lambda_1 = \emptyset $ then define $ A:\mathcal{B}\longrightarrow \mathbb{H} $ as follows:
	\begin{eqnarray*}
		Ae_n&=&Te_{n+1},~\forall ~ n\in \mathbb{N}\\
		Ae_\alpha&=&0,~\forall ~\alpha \in \Lambda_2.
	\end{eqnarray*}
	Proceeding as before, we can show that $ T \parallel A$ but there exists no $x \in M_T \cap M_A$ such that $ Tx \parallel Ax.$ Therefore, $H_0$ must be finite-dimensional subspace of $\mathbb{H}$.\\
	
	Next, we show that $\|T\|_{H_0^{\perp}}< 1$. If possible, suppose that $\|T\|_{H_0^{\perp}}= 1$. Then there exists a sequence $\{x_n\}$ in $S_{H_0^{\perp}}$ such that $\|Tx_n\|\longrightarrow 1$. Define $A:\mathbb{H}\longrightarrow \mathbb{H}$ by $Az=Ty,$ where $z=x+y,~x\in H_0,~y\in H_0^{\perp}$.
	Now, $\|z\|=\|x+y\|=1\Longrightarrow \|y\|\leq 1$. Thus, $\|Az\|=\|Ty\|\leq 1$. Again, $\|Ax_n\|=\|Tx_n\|\longrightarrow 1$ gives that $\|A\|=1$. Now, $\|(T+A)x_n\|=2\|Tx_n\|\longrightarrow 2$. Therefore, $2 \leq \|T+A\|\leq \|T\|+\|A\|=2$ and so $\|T+A\|=\|T\|+\|A\|$, i.e., $T \parallel A$. From the construction of $A$ it follows that $ Ax =0, $ if $ x \in M_T$. Since $A$ is non-zero it follows that  $M_T \cap M_A =\emptyset $ and so there does not exist any $x \in M_T \cap M_A$ such that $Tx \parallel Ax$. This is in clear contradiction with the hypothesis. Therefore,  we must have, $\|T\|_{H_0^{\perp}}< 1$. This establishes the theorem.
\end{proof}

\begin{remark}
We note   that $T\parallel A$ does not imply $T\parallel^{\epsilon}A$ with $\epsilon=0.$ This difference between $T\parallel A$ and $T\parallel^{\epsilon}A$ with $\epsilon=0$ justifies the fact that conditions $(i)$ and $(ii)$ in Theorem \ref{theorem:epsilon} are not equivalent for $\epsilon=0$, whereas conditions $(i)$ and $(ii)$ in Theorem \ref{theorem:hil} are equivalent.
\end{remark}

\section{Birkhoff-James orthogonality of bounded linear operators}
We begin this section with an easy proposition on approximate Birkhoff-James orthogonality ($\perp_D^{\epsilon}$).
\begin{prop}
		Let $\mathbb{X},~\mathbb{Y}$ be two normed linear spaces. Let $T \in\mathbb{B}({\mathbb{X},\mathbb{Y}})$ and $x \in M_T$. Then for any $\epsilon \in [0,1)$ and $y \in \mathbb{X}$, $Tx \perp_D^{\epsilon} Ty \Longrightarrow x \perp_D^{\epsilon} y$. 
\end{prop}
\begin{proof}
	Let $Tx \perp_D^{\epsilon}Ty$. Then for any $\lambda\in K$, $\|Tx+\lambda Ty\|\geq \sqrt{1-{\epsilon}^2} \|Tx\|$. Therefore, for any $\lambda \in K$, $\|T\|\|x+\lambda y\|\geq \|Tx+\lambda Ty\| \geq \sqrt{1-{\epsilon}^2} \|Tx\|=\sqrt{1-{\epsilon}^2} \|T\|\|x\|$. Thus, for any $\lambda \in K$, $\|x+\lambda y\|\geq \sqrt{1-{\epsilon}^2}\|x\|$. Therefore, $x \perp_D^{\epsilon}y$. 
\end{proof}

In \cite[Th. 2.1]{J}, James characterized Birkhoff-James orthogonality in terms of linear functionals. In the next lemma, we characterize approximate Birkhoff-James orthogonality ($\perp_D^{\epsilon}$) in terms of linear functionals.
\begin{lemma}
	Let $\mathbb{X}$ be a normed linear space and $x,y\in\mathbb{X}$. Then for any $\epsilon \in [0,1)$, $x\perp_D^{\epsilon}y$ if and only if  there exists a linear functional $f \in S_{\mathbb{X}^*}$ such that $f(x)\geq \sqrt{1-{\epsilon}^2} \|x\|$ and $f(y)=0$.
\end{lemma}
\begin{proof}
The necessary part follows from a slight variation of \cite[Cor. 6.8]{Con}. Let us prove the sufficient part of the lemma. Suppose there exists a linear functional $f \in S_{\mathbb{X}^*}$ such that $f(x)\geq \sqrt{1-{\epsilon}^2} \|x\|$ and $f(y)=0$. Then for any $\lambda \in K$, $\sqrt{1-{\epsilon}^2} \|x\|\leq f(x)=f(x+\lambda y)\leq \|f\|\|x+\lambda y\|=\|x+\lambda y\|$. Therefore, $x\perp_D^{\epsilon}y$.
\end{proof}

The idea of the following theorem is adapted from \cite[Th. 2.1]{J}. However, for the sake of completeness, we present a complete proof of the result.

\begin{theorem}
	Let $\mathbb{X}$ be a normed linear space. Let $x$ be a nonzero element in $ \mathbb{X}$ and $H$ be a hyperspace in $\mathbb{X}$. Then for $\epsilon \in [0,1)$, $x\perp_D^{\epsilon}H$ if and only if there exists $f\in S_{\mathbb{X}^*}$ such that $f(x)\geq \sqrt{1-{\epsilon}^2}\|x\|$, where $H=ker(f)$.
\end{theorem}
\begin{proof}
	First suppose that there exists $f\in S_{\mathbb{X}^*}$ such that $f(x)\geq \sqrt{1-{\epsilon}^2}\|x\|$, where $H=ker(f)$. Then for any $h\in H$ and for any $\lambda \in K$, $$\sqrt{1-{\epsilon}^2}\|x\| \leq f(x)=f(x+\lambda h)\leq \|f\|\|x+\lambda h\|=\|x+\lambda h\|.$$ This implies that $\sqrt{1-{\epsilon}^2}\|x\|\leq\|x+\lambda h\|$ for any $h\in H$ and for any $\lambda \in K$. Thus, $x\perp_D^{\epsilon}H$.\\
	Conversely, suppose that $x\perp_D^{\epsilon}H$. Then it is easy to check that $\mathbb{X}=span\{x,H\}$. Define $g:\mathbb{X}\longrightarrow K$ by $g(ax+h)=a$, where $a\in K$ and $h\in H$. Clearly $g$ is linear and $ker(g)=H$. Now, $|g(ax+h)|=|a|\leq\frac{\|ax+h\|}{\sqrt{1-{\epsilon}^2}\|x\|}$, since $x\perp^{\epsilon}_D h$. Thus, $\|g\|\leq \frac{1}{\sqrt{1-{\epsilon}^2}\|x\|}\Rightarrow \sqrt{1-{\epsilon}^2}\|g\|\|x\|\leq 1=g(x)$. Let $f=\frac{1}{\|g\|}g$. Then clearly, $f\in S_{\mathbb{X}^*}$, $f(x)\geq \sqrt{1-{\epsilon}^2}\|x\|$ and $H=ker (g)=ker (f)$. This establishes the theorem.
\end{proof}

Our next objective is to study Birkhoff-James orthogonality in the space of bounded linear operators by classifying them into two different cases:  ``$T\perp_B A$ but $T\not\perp_{SB} A$'' and ``$T\perp_{SB} A$''. In \cite[Th. 2.1]{SP}, Sain and Paul proved that if $T$ is a bounded linear operator on a finite-dimensional real Banach space $\mathbb{X}$ and $D$ is a non empty connected subset of $S_{\mathbb{X}}$ such that $M_T=D\cup(-D)$, then for any $A\in \mathbb{B}(\mathbb{X})$, $T\perp_B A$ if and only if there exists $x\in D$ such that $Tx\perp_B Ax$. Note that, if there exists $x\in D$ such that $Tx\perp_{SB}Ax$ then for any $\lambda \in K\setminus \{0\}$, $\|T+\lambda A\|\geq \|(T+\lambda A)x\|>\|Tx\|=\|T\|$, i.e., $T\perp_{SB}A.$ However, the following example illustrates that if $T\perp_{SB}A$ then there may not exist any $x\in D$ such that $Tx\perp_{SB}Ax$.
\begin{example}\label{example:DU-D}
	Consider the two-dimensional real normed linear space $\mathbb{X}$ such that $S_{\mathbb{X}}=\{(x,y)\in\mathbb{R}^2:\mid x\mid=1,~\mid y\mid\leq 1\}\cup\{(x,y)\in\mathbb{R}^2:x^2+(y-1)^2=1, ~1\leq y\leq 2\}\cup\{(x,y)\in\mathbb{R}^2:x^2+(y+1)^2=1, ~-2\leq y\leq -1\} $.  Define $T,~A\in\mathbb{B}(\mathbb{X})$ by $T(1,1)=(1,1),~T(0,2)=(0,0)$ and $A(1,1)=(0,2),~A(0,2)=(1,-1)$. Now, it is easy to observe that $\|T\|=1$ and $M_T=D\cup(-D)$, where $D=\{(1,y):-1\leq y\leq 1\}$. Let $\lambda >0$. Then $\|T+\lambda A\|\geq \|(T+\lambda A)(1,1)\|=\|(1,1+2 \lambda)\|>1=\|T\|$. Next, let $\lambda <0$. Then $\|T+\lambda A\|\geq \|(T+\lambda A)(1,0)\|=\|(1-\frac{\lambda}{2},1+\frac{5 \lambda}{2})\|>1=\|T\|$. Therefore, $T\perp_{SB}A$. Now, let $(1,y)\in D$. Then $T(1,y)=(1,1)$. Clearly, there does not exist any $(x_1,y_1)\in \mathbb{X}$ such that $(1,1)\perp_{SB}(x_1,y_1)$. Therefore, $T(1,y)\not\perp_{SB}A(1,y)$ for any $(1,y)\in D$.
\end{example} 

In the following theorem, we characterize strong Birkhoff-James orthogonality ($T\perp_{SB}A$) in the space of all bounded linear operators between finite-dimensional Banach spaces.
\begin{theorem}
	Let $\mathbb{X},~\mathbb{Y}$ be finite-dimensional Banach spaces. Let $T,~A \in \mathbb{B}(\mathbb{X},\mathbb{Y})$. Then $T \perp_{SB}A$ if and only if for any $\epsilon>0$, there exists $\lambda_{\epsilon}>0$ such that for each $|\lambda|<\lambda_{\epsilon}$, there exists $y_\lambda \in (\cup_{x\in M_T}B(x,\epsilon))\cap S_{\mathbb{X}}$ such that $\|Ty_\lambda +\lambda Ay_\lambda\|>\|T\|$.
\end{theorem}
\begin{proof}
	First, we prove the easier sufficient part of the theorem. \\
	Let $\epsilon>0$ be given. Then there exists $\lambda_{\epsilon}>0$ such that for each $|\lambda|<\lambda_{\epsilon}$, there exists $y_\lambda \in (\cup_{x\in M_T}B(x,\epsilon))\cap S_{\mathbb{X}}$ such that $\|Ty_\lambda +\lambda Ay_\lambda\|>\|T\|$. Therefore, $\|T+\lambda A\|\geq \|Ty_\lambda+\lambda Ay_\lambda\|>\|T\|$. Now, by the convexity of the norm function, it is easy to observe that $T\perp_{SB}A$.\\
	For the necessary part, suppose that $T\perp_{SB}A$. Let $\epsilon >0$ be given. Now, using the compactness of $S_{\mathbb{X}}$, it is easy to observe that $\sup\{\|Tz\|:z\in S_{\mathbb{X}}\setminus (\cup_{x\in M_T}B(x,\epsilon))\}<\|T\|-\delta$ for some $\delta >0$. Let $z\in S_{\mathbb{X}}\setminus (\cup_{x\in M_T}B(x,\epsilon))$. Then $\|Tz\|<\|T\|-\delta$. Therefore, 
	\begin{eqnarray*}
		\|Tz+\lambda Az\|&\leq& \|Tz\|+|\lambda| \|Az\|\\
		& <& \|T\|-\delta +|\lambda|\|A\|\\
		& <& \|T\|~~~~(for ~some ~ |\lambda|<\lambda_{\epsilon}).
	\end{eqnarray*}	
	Now, since $\|T+\lambda A\|>\|T\|$, $(T+\lambda A)$ does not attain norm in $S_{\mathbb{X}}\setminus (\cup_{x\in M_T}B(x,\epsilon)$ for each $|\lambda|<\lambda_{\epsilon}$. Hence there exists $y_\lambda \in S_{\mathbb{X}}\cap (\cup_{x\in M_T}B(x,\epsilon))$ such that $\|Ty_\lambda +\lambda Ay_\lambda\|=\|T+\lambda A\|>\|T\|$. 
	This completes the proof of the theorem. 
\end{proof}  
 
Now, we study the case when ``$T\perp_B A$ but $T\not\perp_{SB}A$'', where $T$ and $A$ are compact linear operators defined from a reflexive Banach space to any normed linear space.
\begin{theorem}\label{theorem:compactortho}
	 Let $\mathbb{X}$ be a reflexive Banach space and $\mathbb{Y}$ be any normed linear space. Let $T,~A\in \mathbb{K}(\mathbb{X},\mathbb{Y})$ be such that $T\perp_B A$ but $T\not\perp_{SB} A$. Then there exists $x\in M_T$ such that $Tx\perp_B Ax$.
\end{theorem}
\begin{proof}
	Without loss of generality, we may and do assume that $\|T\|=1$.
	Since $T\perp_B A$ but $T\not \perp_{SB} A$, there exists $\mu \in K\setminus \{0\}$ such that $\|T+\mu A\|=\|T\|=1$. Let $B=-\mu A -T$. Then $\|B\|=1.$ Now, $T\perp_B A$ gives that $T\perp_B (B+T)$. Therefore, from \cite[Th. 2.1]{J}, we have, there exists $F\in \mathbb{K}(\mathbb{X},\mathbb{Y})^*$ such that $\|F\|=1,~ F(T)=\|T\|$ and $F(T+B)=0$. Clearly, $F(-B)=\|T\|=1=\|-B\|$. Therefore, $\|T\|+\|-B\|=F(T)+F(-B)=F(T-B)\leq \|F\|\|T-B\|=\|T-B\|\leq\|T\|+\|-B\|$. Thus, $\|T-B\|=\|T\|+\|-B\|$. Now, by Remark \ref{remark:compact}, we have, there exists $x \in M_T\cap M_B $ such that $\|Tx-Bx\|=\|Tx\|+\|-Bx\|$. We note that there exists $g\in S_{\mathbb{X}^*}$ such that $g(Tx-Bx)=\|Tx-Bx\|$. Therefore, $g(Tx)+g(-Bx)=g(Tx-Bx)=\|Tx-Bx\|\leq \|Tx\|+\|-Bx\|$. Again, $g(Tx)\leq \|Tx\|$ and $g(-Bx)\leq \|-Bx\|$ gives that $g(Tx)=\|Tx\|=1$ and $g(-Bx)=\|-Bx\|=1$. This implies that $g(Tx+Bx)=0$. Therefore, by \cite[Th. 2.1]{J}, we have, $Tx\perp_B (Tx+Bx)$, i.e., $Tx\perp_B (-\mu Ax)$. Hence $Tx \perp_B Ax$. This completes the proof of the theorem.    
\end{proof}
\begin{remark}
 From the above theorem and \cite[Th. 2.1]{PSG}, it is clear that for $T,A\in\mathbb{K}(\mathbb{X},\mathbb{Y})$, where $\mathbb{X}$ is a reflexive Banach space and $\mathbb{Y}$ is any normed linear space, $T\perp_B A$ implies there exists $x\in M_T$ such that $Tx \perp_B Ax$ if either $T\not\perp_{SB}A$ or $M_T=D\cup(-D)$($D$ is a non-empty compact connected subset of $S_{\mathbb{X}}$).
\end{remark}
Combining Theorem \ref{theorem:compactortho} of this paper and \cite[Th. 2.3]{SPH}, we obtain the following corollary.
\begin{corollary}\label{corollary:2dim}
Let $\mathbb{X}$ be a two-dimensional real normed linear space. Let $T\in \mathbb{B}(\mathbb{X})$. If for any $A\in \mathbb{B}(\mathbb{X})$, $T\perp_B A$ implies that $T \not\perp_{SB} A$ then $M_T$ cannot have more than two components.
\end{corollary}
\begin{proof}
Let $T\in \mathbb{B}(\mathbb{X})$. Suppose for any $A\in \mathbb{B}(\mathbb{X})$, $T\perp_B A$ implies that $T \not\perp_{SB} A$. If possible, suppose that $M_T$ has more than two components. Then by \cite[Th. 2.3]{SPH}, we have, there exists $A\in \mathbb{B}(\mathbb{X})$ such that $T\perp_B A$ but $Tx \not\perp_B Ax$ for any $x \in M_T$. Therefore, by Theorem \ref{theorem:compactortho}, we have $T\perp_{SB}A$, which contradicts the hypothesis of the corollary. Thus, $M_T$ cannot have more than two components.
\end{proof}
The following example shows that the converse of Corollary \ref{corollary:2dim} is not true. 
\begin{example}
Consider $T,A\in \mathbb{B}(\ell_{\infty}^2)$ defined by $T(x,y)=(x,x)$ and $A(x,y)=(-x,x)$ respectively. Clearly, $M_T$ has only two components but $T\perp_{SB}A$, since $\|T+\lambda A\|\geq \|(T+\lambda A)(1,1)\|>\|T(1,1)\|=\|T\|$ for all $\lambda \in\mathbb{R}\setminus \{0\}$. 
\end{example}
We further study the case ``$T\perp_B A$ but $T\not\perp_{SB}A$'', where $T$ and $A$ are bounded linear operators defined between any two normed linear spaces. For this, we need the following lemma.  
 \begin{lemma}\label{lemma:sequence}
 Let $\mathbb{X}$ be a normed linear space. Suppose $\{x_n\},~\{y_n\}$ are two bounded sequences of $\mathbb{X}$ such that $$\lim_{n\to\infty}\|x_n +y_n\|=\lim_{n\to\infty}\|x_n\|+\lim_{n\to\infty}\|y_n\|.$$ Then for each $k \in \mathbb{N},$ there exists $f_{n_k}\in S_{\mathbb{X}^*}$ such that $$\lim_{k\to\infty}f_{n_k}(x_{n_k})=\lim_{k\to\infty}\|x_{n_k}\|$$ and  $$\lim_{k\to\infty}f_{n_k}(y_{n_k})=\lim_{k\to\infty}\|y_{n_k}\|.$$	
\end{lemma}

\begin{proof}
	By Hahn-Banach Theorem, we have, for each $n \in \mathbb{N}$, there exists $f_n \in S_{\mathbb{X}^*}$ such that $f_n(x_n+y_n)=\|x_n+y_n\|$. Therefore, $\lim_{n\to\infty}f_n(x_n+y_n)=\lim_{n\to\infty}\|x_n+y_n\|$. Since $\{x_n\},~\{y_n\}$ are bounded sequences of $\mathbb{X}$,  $\{f_n(x_n)\},~\{f_n(y_n)\}$ are bounded sequences of $\mathbb{R}$. This proves that $\{f_n(x_n)\},~\{f_n(y_n)\}$ have convergent subsequences. Let $\lim_{k\to\infty}f_{n_k}(x_{n_k})$ and $\lim_{k\to\infty}f_{n_k}(y_{n_k})$ exists. Then 
	\begin{eqnarray*}
	\lim_{k\to\infty}f_{n_k}(x_{n_k})+\lim_{k\to\infty}f_{n_k}(y_{n_k})&=&\lim_{k\to\infty}f_{n_k}(x_{n_k}+y_{n_k})\\
	&=& \lim_{k\to\infty}\|x_{n_k}+y_{n_k}\|\\
	&=& \lim_{k\to\infty}\|x_{n_k}\|+\lim_{k\to\infty}\|y_{n_k}\|.
	\end{eqnarray*}
Again, since $\lim_{k\to\infty}f_{n_k}(x_{n_k})\leq \lim_{k\to\infty}\|x_{n_k}\|$ and $\lim_{k\to\infty}f_{n_k}(y_{n_k}) \leq \lim_{k\to\infty}\|y_{n_k}\|$, we have, $$\lim_{k\to\infty}f_{n_k}(x_{n_k})=\lim_{k\to\infty}\|x_{n_k}\|$$ and  $$\lim_{k\to\infty}f_{n_k}(y_{n_k})=\lim_{k\to\infty}\|y_{n_k}\|.$$	This completes the proof of the lemma. 
\end{proof}

Let us now obtain the desired necessary condition for ``$T\perp_B A$ but $T\not\perp_{SB}A$'', by applying the previous lemma.
\begin{theorem}\label{theorem:boundedortho}
	Let $\mathbb{X},~\mathbb{Y}$ be two normed linear spaces and $T,~A \in \mathbb{B}(\mathbb{X},\mathbb{Y})$. Suppose $T\perp_B A$ but $T\not\perp_{SB}A$. Then either there exists a sequence $\{x_n\}$ in $S_{\mathbb{X}}$ such that $\|Tx_n\|\longrightarrow \|T\|,~ Ax_n \longrightarrow 0$  or there exists a sequence $\{x_n\}$ in $S_{\mathbb{X}}$ and a sequence $\{\epsilon_n\}$ in $\mathbb{R}^+$  such that $\|Tx_n\|\longrightarrow \|T\|,~ \epsilon_n \longrightarrow 0$ and $Tx_n \perp_D^{\epsilon_n} Ax_n$. 
\end{theorem}
\begin{proof}
	Without loss of generality assume that $\|T\|=1$. Since $T\perp_B A$ but $T\not\perp_{SB}A$, there exists $\mu \in K\setminus \{0\}$ such that $\|T+\mu A\|=\|T\|=1$.  Let $B=-\mu A -T$. Then $\|B\|=1.$ Now, $T\perp_B A$ gives that $T\perp_B (B+T)$. Therefore, from \cite[Th. 2.1]{J}, we have, there exists $F\in \mathbb{B}(\mathbb{X},\mathbb{Y})^*$ such that $\|F\|=1,~ F(T)=\|T\|$ and $F(T+B)=0$. Clearly, $F(-B)=\|T\|=1=\|-B\|$. Therefore, $\|T\|+\|-B\|=F(T)+F(-B)=F(T-B)\leq \|F\|\|T-B\|=\|T-B\|\leq\|T\|+\|-B\|$. Thus, $\|T-B\|=\|T\|+\|-B\|$. Now, from Theorem \ref{theorem:bounded}, it is easy to observe that there exists a sequence $\{x_n\}$ in $S_{\mathbb{X}}$ such that $$\lim_{n \to \infty}\|Tx_n\|=\|T\|,~\lim_{n \to \infty}\|Bx_n\|=\|B\|$$ and $$\lim_{n\to\infty}\|Tx_n-B x_n\|=\lim_{n\to\infty}\|Tx_n\|+\lim_{n\to\infty}\|-Bx_n\|.$$ Now, by Lemma \ref{lemma:sequence}, without loss of generality we may assume that for each $n \in \mathbb{N}$ there exists $f_n \in S_{\mathbb{Y}^*}$ such that  $$\lim_{n\to\infty}f_n(Tx_n)=\lim_{n\to\infty}\|Tx_n\|=\|T\|=1$$ and $$\lim_{n\to\infty}f_n(-Bx_n)=\lim_{n\to\infty}\|-Bx_n\|=\|B\|=1.$$ Clearly, $f_n(Tx_n)=r_n$ and $f_n(-Bx_n)=s_n$, where $|r_n| \leq 1,~|s_n| \leq 1$, $r_n \longrightarrow 1$ and $s_n \longrightarrow 1$. Therefore, $f_n(Tx_n+Bx_n)=r_n-s_n,$ which converges to $0$ as $n \to\infty$. If $\{Ax_n\}$ has a subsequence converging to $0$ then we are done. So assume that $\{Ax_n\}$ has no subsequence converging to $0$. Without loss of generality we may assume that $\inf_{n} \|Ax_n\|=c>0$. Now, it is easy to observe that $1> 1- (|r_n|-\frac{2}{c|\mu|}|r_n-s_n|)^2\geq 0$, since $|r_n| \to 1,~|r_n|\leq 1$ and $|r_n-s_n| \to 0$. Let for each $n\in\mathbb{N}$,  $\epsilon_n^2=1-(|r_n|-\frac{2}{c|\mu|}|r_n-s_n|)^2$. Then clearly $\epsilon_n \in [0,1)$ for each $n\in \mathbb{N}$ and $\epsilon_n \rightarrow 0$. We show that $Tx_n \perp_D^{\epsilon_n}Ax_n$. First, let $|\lambda|\geq \frac{2}{c}\|Tx_n\|\geq \frac{2}{\|Ax_n\|}\|Tx_n\|$. Then $\|Tx_n+\lambda Ax_n\|\geq |\lambda|\|Ax_n\|-\|Tx_n\|\geq \|Tx_n\|\geq \sqrt{1-\epsilon_n^2}\|Tx_n\|$. Let $|\lambda|<\frac{2}{c}\|Tx_n\|\leq \frac{2}{c}$. Then
	\begin{eqnarray*}
	\|Tx_n+ \lambda Ax_n\|&=&\|Tx_n-\frac{\lambda}{\mu}(Tx_n+Bx_n)\| \\
	&\geq & |f_n\Big (Tx_n-\frac{\lambda}{\mu}(Tx_n+Bx_n)\Big)|\\
	&=&|r_n-\frac{\lambda}{\mu}(r_n-s_n)|\\
	&\geq& |r_n|-\frac{|\lambda|}{|\mu|}|r_n-s_n|\\
	&> & |r_n|-\frac{2}{c|\mu|}|r_n-s_n|\\
	&=& \sqrt{1-\epsilon_n^2}\\
	&\geq& \sqrt{1-\epsilon_n^2}\|Tx_n\|.
	\end{eqnarray*}
	  Hence $Tx_n \perp_D^{\epsilon_n}Ax_n$. This completes the proof of the theorem.
\end{proof}

In the following example we show that the conditions given in Theorem \ref{theorem:boundedortho} are not sufficient to ensure that $T\perp_B A$ but $T\not\perp_{SB}A$.
\begin{example}
Let $\mathbb{X}=\ell_{\infty}(\mathbb{R})$. Define  linear operators $T,A:\ell_{\infty}\longrightarrow \ell_{\infty}$ as follows: 
\begin{eqnarray*}
T(x_1,x_2,x_3,\ldots)& = & (x_1,x_1,x_1,\ldots)\\
A(x_1,x_2,x_3,\ldots)& = & (-x_2,x_2,x_2,\ldots)
\end{eqnarray*}

  Then it is easy to check that $T,A \in \mathbb{B}(\ell_{\infty})$ with  $\|T\|= \|A\|= 1.$ 
	Clearly for  $\lambda \neq 0$, $\|T+\lambda A\|\geq \|(T+\lambda A)(1,1,0,0,0,\ldots)\|=\|(1-\lambda,1+\lambda,1+\lambda,\ldots)\|=1+|\lambda|>1=\|T\|$. Therefore, $T\perp_{SB}A$. But $(1,0,0,\ldots)\in M_T$ and choosing $ y_n = (1,0,0,\ldots) $ we get $\|T{{y_n}}\|\rightarrow \|T\|,~\|A{{y_n}}\|\rightarrow 0.$ \\
	Similarly, defining $B:\ell_{\infty}\longrightarrow \ell_{\infty}$ by $B(x_1,x_2,x_3,\ldots) =  (-x_1,x_1,x_1,\ldots), $ we can check that $T\perp_{SB}A$ although there exists ${y_n}=(1,0,0,\ldots) \in M_T,~ \epsilon_n=0$ such that $\|T{{y_n}}\|\rightarrow \|T\|$ and $T{{y_n}}\perp_D^{\epsilon_n}B{{y_n}}$. 
\end{example}

Now, we turn our attention to the newly introduced notion of semi-rotundity of a normed linear space. As mentioned earlier, every exposed point (strictly convex space)  is a semi-rotund point ( semi-rotund space)  but not conversely, if dimension of the space is greater than two. The notions are equivalent if the dimension of the space is two.
In this context, we first prove the following proposition which states that  every isometry defined between finite-dimensional Banach spaces is a semi-rotund point in the operator space.
\begin{prop}
Let $\mathbb{X},\mathbb{Y}$ be finite-dimensional Banach spaces. Then every isometry from $\mathbb{X}$ to $\mathbb{Y}$ is a semi-rotund point in $ \mathbb{B}(\mathbb{X},\mathbb{Y}). $
\end{prop}
\begin{proof}
Let $T\in \mathbb{B}(\mathbb{X},\mathbb{Y})$ be an isometry, where $\mathbb{X},~\mathbb{Y}$ are finite-dimensional Banach spaces. Clearly, $T$ is invertible. Let $y$ be an exposed point of $\mathbb{B}_{\mathbb{Y}}$. Then there exists $x\in S_{\mathbb{X}}$ such that  $Tx= y.$ Since $y$ is an exposed point of $\mathbb{B}(\mathbb{Y})$, there exists $z\in \mathbb{Y}$ such that $y\perp_{SB}z$. Define $A\in \mathbb{B}(\mathbb{X},\mathbb{Y})$ by $Ax=z$ and $Aw=0$ for all $w\in H,$ where $H$ is a hyperspace such that $x \bot_B H.$  Then $Tx\perp_{SB}Ax$. We also observe that $x\in M_T$, since $T$ is an isometry. Therefore, $T\perp_{SB}A$. Hence, $T$ is a semi-rotund point of $ \mathbb{B}(\mathbb{X},\mathbb{Y}). $
\end{proof}  

In the following theorem, we prove that every non-zero compact linear operator from a reflexive Banach space to a strictly convex Banach space is semi-rotund.

\begin{theorem}\label{theorem:semi-rotund}
	Let $\mathbb{X}$ be a reflexive Banach space and $\mathbb{Y}$ be a strictly convex Banach space. Then every non-zero $T\in \mathbb{K}(\mathbb{X},\mathbb{Y})$ is a semi-rotund point.
\end{theorem}
\begin{proof}
	Since  $\mathbb{X}$ is reflexive and $T$ is compact, $T$ attains its norm. Let $x\in M_T$. By \cite[Cor. 2.2]{J}, there exists $y \in S_{\mathbb{Y}}$ such that $Tx \perp_B y$. Now, since $\mathbb{Y}$ is strictly convex and $ T $ is non-zero, we must have, $Tx\perp_{SB}y$. Define a linear operator $A$ from $\mathbb{X}$ to $\mathbb{Y}$ by $Ax=y$ and $Az=0$ for all $z\in H,$ where $H$ is a hyperspace such that $x \bot_B H.$ Clearly, $A\in \mathbb{K}(\mathbb{X},\mathbb{Y})$. Now, for any $\lambda \in K \setminus\{0\}$, $\|T+\lambda A\|\geq \|(T+\lambda A)x\|=\|Tx+\lambda y\|>\|Tx\|=\|T\|$. Therefore, $T\perp_{SB}A$. This proves that $T$ is a semi-rotund point and completes the proof of the theorem. 
\end{proof}

Finally, using Theorem \ref{theorem:semi-rotund}, we obtain the following corollary that illustrates that semi-rotundity is a strictly weaker property compared to strict convexity.
\begin{corollary}
	Let $\mathbb{X},~\mathbb{Y}$ be finite-dimensional Banach spaces and in addition, let $\mathbb{Y}$ be strictly convex. Then $\mathbb{B}(\mathbb{X},\mathbb{Y})$ is not strictly convex but semi-rotund.
\end{corollary}
\begin{proof}
	Let $x\in S_{\mathbb{X}}$. Define  $T\in \mathbb{B}(\mathbb{X},\mathbb{Y})$ by $Tx=x$ and $Ty=0$ for all $y\in H,$ where $H$ is a hyperspace such that $x \bot_B H.$  Then clearly $\|T\|=1$. Now, $2=\|I\|+\|T\| \geq \|I+T\|\geq \|Ix+Tx\|=2$. This shows that $ \|T\| + \|I\| = \|T + I \|,$ although $T$ and $I$ are linearly independent. Therefore, $\mathbb{B}(\mathbb{X},\mathbb{Y})$ is not strictly convex. Now, being finite-dimensional, $\mathbb{X}$ is reflexive. Since $\mathbb{Y}$ is finite-dimensional, it follows that $\mathbb{B}(\mathbb{X},\mathbb{Y})=\mathbb{K}(\mathbb{X},\mathbb{Y})$. Therefore, by Theorem \ref{theorem:semi-rotund}, we have, every $T\in \mathbb{B}(\mathbb{X},\mathbb{Y})$ is a semi-rotund point. This completes the proof of the corollary.   
\end{proof}

\bibliographystyle{amsplain}

\begin{thebibliography}{99}

\bibitem{AR} L. Aramba\v si\'c, R. Raji\'c, The Birkhoff-James orthogonality in Hilbert C*-modules,  Linear Algebra Appl., 437 (2012) 1913--1929.

\bibitem{BFS} C. Ben\'itez,  M. Fern\'andez  and  M.L. Soriano, \textit{Orthogonality of matrices}, Linear Algebra  Appl., \textbf{422} (2007) 155--163.
\bibitem{BS} R. Bhatia  and P. \v Semrl, \textit{Orthogonality of matrices and distance problem},  Linear Algebra Appl., \textbf{287} (1999)  77--85.


\bibitem{B} G. Birkhoff,
  \textit{Orthogonality in linear metric spaces},
  Duke Math. J. \textbf{1} (1935) 169--172.

\bibitem{BCMWZ} T. Bottazzi, C. Conde, M. S. Moslehian, P. Wojcik and A. Zamani,
  \textit{Orthogonality and parallelism of operators on various Banach spaces},
  J. Austral. Math. Soc. (to appear).
	
	\bibitem{C} J. Chmieli\'nski,
\textit{Linear mappings approximately preserving orthogonality},
J. Math. Anal. Appl. 
\textbf{304} (2005) 158-169.	

\bibitem{Con} J. B. Conway, \textit{A Course in Functional Analysis}, Springer-Verlag, New York, Inc., 1990. 

	\bibitem{D} S.S. Dragomir, \textit{On approximation of continuous linear functionals in normed linear spaces}, An.
Univ. Timi¸soara Ser. ¸Stiin¸t. Mat., \textbf{29} (1991), 51--58.

\bibitem{GSP} P. Ghosh, D. Sain and K. Paul, \textit{On symmetry of Birkhoff-James orthogonality of linear operators},
Adv. Oper. Theory \textbf{2} (2017) 428-434. 
	
	\bibitem{J} R. C. James, 
	  \textit{Orthogonality and linear functionals in normed linear spaces},  
	Trans. Amer. Math. Soc. \textbf{61} (1947) 265-292.
			
			\bibitem{LS} C. K. Li  and H. Schneider, \textit{Orthogonality of matrices},  Linear Algebra Appl., \textbf{47} (2002)  115--122.

		
	\bibitem{MZ} M. S. Moslehian and A. Zamani,
\textit{Exact and approximate operator parallelism}, Canad. Math. Bull.
\textbf{58} (2015) 207-224.	
	
	\bibitem{PSJ} K. Paul, D. Sain and K. Jha,
	    \textit{On strong orthogonality and  strictly convex normed linear spaces},
	    J. Inequal. Appl.
			2013, \textbf{2013:242}.

\bibitem{PSG} K. Paul, D. Sain and P. Ghosh,
\textit{ Birkhoff-James orthogonality and smoothness of bounded linear operators},
Linear Algebra Appl.
\textbf{506} (2016) 551-563.
	
\bibitem{SP} D. Sain  and K. Paul,
  \textit{ Operator norm attainment and inner product spaces}, 
  Linear Algebra Appl.
  \textbf{439} (2013) 2448-2452.
	
	\bibitem{SPH} D. Sain, K. Paul and S. Hait,
  \textit{ Operator norm attainment and Birkhoff-James orthogonality},
 Linear Algebra Appl.
  \textbf{476} (2015) 85-97.
	
	\bibitem{Se} A. Seddik, \textit{Rank one operators and norm of elementary operators}, Linear Algebra Appl. \textbf{424} (2007) 177--183.
	
	\bibitem{BG} T. Bhattacharyya, P. Grover, \textit{Characterization of Birkhoff-James orthogonality}, J. Math. Anal. Appl. 407 (2013) 350--358.
	
	\bibitem{W} P. W\'ojcik, \textit{Orthogonality of compact operators}, Expositiones Mathematicae, 35 (2017) 86-94
	
	\bibitem{ZM} A. Zamani and M. S. Moslehian,
	\textit{Norm-Parallelism in the geometry of Hilbert $C^*$-modules}, Indag. Math. (N.S.) \textbf{27} (2016) 266-281.
	
	
\end{thebibliography}

\end{document}